\documentclass[11pt,english]{article}
\usepackage[margin= 2 cm,bottom=23mm, top= 20mm]{geometry}

\usepackage{amsthm}
\usepackage{amsmath}
\usepackage{amssymb}
\usepackage{setspace}
\usepackage{mathtools}
\usepackage{verbatim}
\usepackage{csquotes}
\usepackage[hidelinks]{hyperref}
\usepackage{thm-restate}
\usepackage{cleveref}
\usepackage{enumitem}
\setlist[itemize]{leftmargin=*} 
\setlist[enumerate]{leftmargin=*}
\usepackage{framed}
\usepackage{floatrow}
\usepackage[T1]{fontenc}
\usepackage{bbm}
\usepackage[color=Orange!50!white, textwidth=22mm]{todonotes}

\usepackage{mathdots}
\usepackage{xcolor}
\usepackage{diagbox}
\usepackage{colortbl}

\usepackage{graphicx}
\usepackage{subcaption}

\theoremstyle{plain}

\newtheorem{theorem}{Theorem}[section]
\newtheorem{claim}[theorem]{Claim}

\newtheorem{lemma}[theorem]{Lemma}

\theoremstyle{definition}

\newtheorem{defn}[theorem]{Definition}
\newtheorem*{defn*}{Definition}

\expandafter\def\expandafter\normalsize\expandafter{%
    \normalsize
    \setlength\abovedisplayskip{4pt}
    \setlength\belowdisplayskip{4pt}
    \setlength\abovedisplayshortskip{4pt}
    \setlength\belowdisplayshortskip{4pt}
}

\usepackage[square,sort,comma,numbers]{natbib}
\newcommand{\calF}{\mathcal{F}}

\newcommand{\calG}{\mathcal{G}}
\newcommand{\calH}{\mathcal{H}}

\newcommand{\calT}{\mathcal{T}}

\def\eps {\varepsilon}

\newcommand{\Bin}{\mathrm{Bin}}

\newcommand{\bit}{\mathrm{bit}}

\newcommand{\tw}{\mathrm{tw}}

\renewcommand{\Pr}{\mathbb{P}}

\DeclareMathOperator*{\argmax}{arg\,max}

\newcommand{\hide}[1]{}

\title{Lower bounds for Ramsey numbers of bounded degree hypergraphs}
\author{Domagoj Brada\v{c}\thanks{Department of Mathematics, ETH, Z\"urich, Switzerland. Research supported in part by SNSF grant 200021-228014. Email: \textbf{\{domagoj.bradac, zach.hunter, benjamin.sudakov\}@math.ethz.ch}.}
\and Zach Hunter\footnotemark[1] \and Benny Sudakov\footnotemark[1]}
\date{}

\begin{document}

\maketitle

\begin{abstract}
    We prove that, for all $k \ge 3,$ and any integers $\Delta, n$ with $n \ge \Delta,$ there exists a $k$-uniform hypergraph on $n$ vertices with maximum degree at most $\Delta$ whose $4$-color Ramsey number is at least $\tw_k(c_k \Delta) \cdot n$, for some constant $c_k > 0$, where $\tw_k$ denotes the tower function. For $k \ge 4,$ this is tight up to the constant $c_k$ and for $k = 3$ it is known to be tight up to a factor of $\log \Delta$ on top of the tower. Our bound extends a well-known result of Graham, R\"{o}dl and Ruci\'{n}ski for graphs and answers a question of Conlon, Fox and Sudakov from 2008.
\end{abstract}

\section{Introduction}
Given a positive integer $q$ and a $k$-uniform hypergraph (or $k$-graph for short) $H$, the $q$-color Ramsey number of $H$, denoted by $r(H; q)$ is the minimum integer $N$ such that in any $q$-coloring of the complete $k$-uniform hypergraph on $N$ vertices, there is a monochromatic copy of $H$. The existence of these numbers is a cornerstone result of Ramsey~\cite{ramsey} and since then a great deal of work has been done on establishing good quantitative bounds for these numbers.

The most classical problem is when $H$ is a graph clique. Famous results of Erd\H{o}s~\cite{erdos} and Erd\H{o}s and Szekeres~\cite{erdos-szekeres} show that $\sqrt{2}^n < r(K_n; 2) < 4^n$. The upper bound was improved in a recent breakthrough of Campos, Griffiths, Morris and Sahasrabudhe~\cite{camposgriffiths} (see also~\cite{balister} and ~\cite{best-ramsey}). For hypergraph cliques, the picture is somewhat less clear. Erd\H{o}s and Rado~\cite{erdos-rado} showed that $r(K_n^{(k)}; q) \le \tw_k(O_q(n))$, where $\tw_k$ denotes the tower function defined as $\tw_1(x) = x$ and $\tw_k(x) = 2^{\tw_{k-1}(x)}$ for $k \ge 2.$ On the other hand, an ingenious construction of Erd\H{o}s and Hajnal, known as the stepping-up lemma (see e.g. \cite{graham1991ramsey}), shows that for $k \ge 3$, $r(K_n^{(k)}; 2) \ge \tw_{k-1}(\Omega_k(n^2))$ and $r(K_n^{(k)}; 4) \ge \tw_k(\Omega_k(n))$. Notably, for at least $4$ colors, the lower bound matches the upper bound up to the constant on top of the tower, and it is a major open problem to close the gap for two colors.

One prominent direction has been to understand the growth rate of Ramsey numbers of sparse graphs and hypergraphs. As an early application of the regularity method, Chv\'{a}tal, R\"{o}dl, Szemer\'{e}di and Trotter~\cite{chvatal} proved that Ramsey numbers of bounded degree graphs are linear in their number of vertices. In other words, for any $\Delta$ and $q$ there is a constant $C = C(\Delta, q)$ such that for any graph $G$ with $n$ vertices and maximum degree at most $\Delta,$ it holds that $r(G; q) \le C \cdot n$. Because their proof relies on the regularity lemma, it gives a tower-type dependence of $C$ on $\Delta$. Since this breakthrough result, the problem of determining the correct dependence of $C$ on $\Delta$ has received a lot of attention. 

Eaton~\cite{eaton} used the so-called weak regularity lemma of Duke, Lefmann and R\"{o}dl~\cite{duke-lefmann-rodl} to improve the bound to $C(\Delta, q) \le 2^{2^{c' \Delta}}$. For two colors, Graham, R\"{o}dl and Rucinski~\cite{graham2000graphs} pioneered the so-called greedy embedding approach which avoids the use of the regularity lemma and gave the bound $C(\Delta, 2) \le 2^{c' \Delta \log^2 \Delta}$. Furthermore, they showed that for any $\Delta, n$ there exists a graph $G$ on $n$ vertices with maximum degree at most $\Delta$ satisfying $r(G; 2) \ge 2^{c'' \Delta} \cdot n$ for some absolute constant $c''$. Finally, Conlon, Fox and Sudakov~\cite{conlonfoxsudakovdeltalogdelta} built on the approach of Graham, R\"{o}dl and Rucinski to prove the current best bound $C(\Delta, 2) \le 2^{c' \Delta \log \Delta}$. For more than two colors, the best known bound from \cite{FS} is weaker $C(\Delta, q) \le 2^{O_q(\Delta^2)}$ and it is achieved using dependent random choice. 

Next we turn our attention to Ramsey numbers of bounded degree hypergraphs. The result that bounded degree hypergraphs have linear Ramsey numbers was proved by Cooley, Fountoulakis, K\"{u}hn and Osthus~\cite{cooleyfount3, cooleyfount-all-k} and for $3$-uniform hypergraphs independently by Nagle, Olsen, R\"{o}dl and Schacht~\cite{nagleolsen}. Formally, they show that for any $k, \Delta, q$, there exists a constant $C = C^{(k)}(\Delta, q)$ such that $r(H; q) \le Cn$ for any $n$-vertex $k$-graph with maximum degree $\Delta$. These results relied on the hypergraph regularity method and thus gave Ackermann type bounds on $C$ with respect to $\Delta$. Conlon, Fox and Sudakov~\cite{ramseynumbersofsparsehypergraphs} used dependent random choice to show that for $k \ge 4$, $C^{(k)}(\Delta, q) \le \tw_{k}(c \Delta)$ for some $c = c(k, q)$, while for $k = 3,$ their argument yields $C^{(3)}(\Delta, q) \le \tw_{3}(c' \Delta \log \Delta)$, for some $c' = c'(q)$.

This bound was shown to be tight in some cases. Specifically, Conlon, Fox and Sudakov~\cite{ramseynumbersofsparsehypergraphs} constructed $3$-uniform hypergraphs with maximum degree $\Delta$ and $O(\Delta)$ vertices, whose $4$-color Ramsey number is at least $\tw_3(c \Delta)$ for some constant $c>0$, and this was generalized to higher uniformities by Brada\v{c}, Fox and Sudakov~\cite{bradavc2023ramsey}. For later use, we recall this result as follows:

\begin{theorem}{\cite[Theorem~1.3]{bradavc2023ramsey}}\label{bounded construction}
    For any $k\ge 2$, there is a constant $C_k$ so that for any $n \ge C_k$, there exists a $k$-uniform $n$-vertex hypergraph $H$ with maximum degree at most $C_k n$ whose $4$-color Ramsey number is at least $\tw_k(n/C_k)$.
\end{theorem}

The upper bound on the maximum degree is not explicitly stated in~\cite{bradavc2023ramsey} but it can be easily extracted from the proof. Notably, however, in both of these constructions, the number of vertices of the hypergraph is comparable to its maximum degree. Hence, it was an open problem (mentioned in \cite{ramseynumbersofsparsehypergraphs, bradavc2023ramsey}) to obtain constructions with large but fixed $\Delta$ and growing $n$.

Our main result achieves this. It can be viewed as a generalization of these results and the aforementioned lower bound of Graham, R\"{o}dl and Rucinski for Ramsey numbers of bounded degree graphs.

\begin{theorem} \label{thm:main}
    For any $k \ge 2,$ there is a constant $c_k > 0$ such that for any integers $\Delta \ge 1/c_k$ and  $n\ge \Delta$, there exists a $k$-uniform $n$-vertex hypergraph $H$ with maximum degree at most $\Delta$ whose $4$-color Ramsey number is at least $\tw_k(c_k \Delta) \cdot n$.
\end{theorem}

For $k \ge 4,$ the result is optimal up to the constant $c_k$, matching the aforementioned $\tw_k(O_k(\Delta))\cdot n$ upper bound of Conlon, Fox and Sudakov \cite{ramseynumbersofsparsehypergraphs}. For $k = 3,$ the best known upper bound, also from~\cite{ramseynumbersofsparsehypergraphs}, is of the form $\tw_3(c_3 \Delta \log \Delta) \cdot n$. However, we believe our result is tight for $k=3$ as well. As it relies on a variant of the stepping-up procedure, our construction requires four colors. 

One of the key novelties of our approach is taking two different constructions whose edge union has large Ramsey number. Roughly speaking, the hypergraph we construct consists of two parts: a random part, which behaves similarly as the construction of Graham, R\"{o}dl and Rucinski and acts as a base case; and a more structured part which interacts with the stepping up coloring in a way that allows for an inductive step which, essentially, reduces the uniformity by one while keeping the properties of the hypergraph we work with. We shall greatly elaborate on this in Subsection~\ref{subsec:outline} once we introduce the stepping up coloring we use.

\section{Proof of Theorem~\ref{thm:main}}

\subsection{Setup}
To begin, we recall an important function used in this construction. For a nonnegative integer $x,$ let $x = \sum_{i=0}^{\infty} a_i 2^i$ be its unique binary representation (where $a_i = 0$ for all but finitely many $i$). For $i \ge 1,$ we denote $\bit(x, i) = a_{i-1}.$ For distinct $x, y \ge 0,$ we define $\delta(x, y) \coloneqq \max \{ i \in \mathbb{Z}_{>0} \, \vert \, \bit(x, i) \neq \bit(y, i)\}.$ Additionally, for convenience we define $\delta(x, x) = 0,$ for all $x \in \mathbb{Z}_{\ge 0}.$ For nonnegative integers $x_1 \le x_2 \le \dots \le x_t,$ we denote $\delta(\{x_1, \dots, x_t\}) = (\delta_1, \dots, \delta_{t-1})$ where for $i \in [t-1],$ $\delta_i = \delta(x_i, x_{i+1}).$ The following properties of this function are well known and easy to verify.

\begin{enumerate}[label=P\arabic*)]
    \item \label{prop:delta-smaller} For distinct $x, y$ we have $x < y \iff \bit(x, \delta(x, y)) < \bit(y, \delta(x, y))$.
    \item \label{prop:not-equal} For any $x \le y \le z$ with $x < z$, $\delta(x, y) \neq \delta(y, z)$.
    \item \label{prop:maximum} For any $x_1 \le x_2 \le \dots \le x_k,$ $\delta(x_1, x_k) = \max_{1 \le i \le k-1} \delta(x_i, x_{i+1})$. 
\end{enumerate}

To help illustrate the function $\delta,$ we borrowed and slightly modified Figure~\ref{fig:deltas} from~\cite{growth-rate}.

 \begin{figure}[ht]
    \centering
    \includegraphics[scale=0.9]{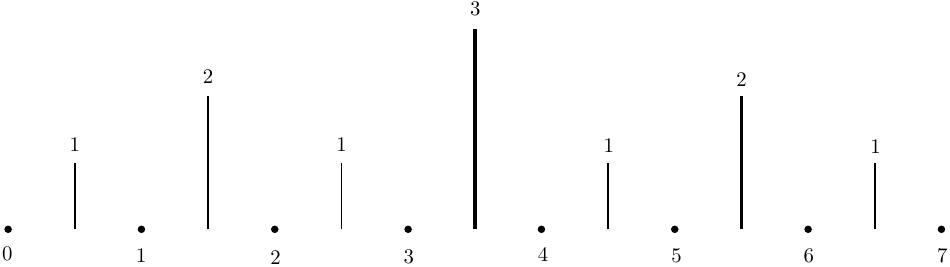}
    \caption{It is convenient to think about the function $\delta$ in the following way. The value of $\delta(x, y)$ is given by the highest line between $x$ and $y$ on the picture. So, for example, $\delta(0, 1) = \delta(6, 7) = 1,$ $\delta(0, 3) = \delta(5, 6) = 2, \, \delta(3, 4) = \delta(2, 7) = 3.$}
    \label{fig:deltas}
\end{figure}

Let us now define the coloring which will be used to prove Theorem~\ref{thm:main}. Given a positive integer $m$, define $M_2(m) = 2^{10^{-8}m}$ and for $k=3,4,\dots$ we further define $M_{k}(m)= 2^{M_{k-1}(m)-1}$. For $k \ge 2,$ we define a coloring $\phi^{(k)}_m$ of all multisets of $k$-elements from $[0, M_k(m))$ as follows. For $k=2,$ the coloring $\phi^{(2)}_m$ is essentially a random coloring. Formally, the coloring of distinct pairs satisfies the properties given by Lemma~\ref{lem:graham-rodl-rucinski} and we additionally define $\phi^{(2)}_m(x, x) = \mathrm{red}$ for all $x \in [0, M_2(m))$. 

For $k \ge 3,$ the coloring $\phi_m^{(k)}$ is defined as follows. For a multiset $\{x_1, \dots, x_k\}$ with $0 \le x_1 \le \dots \le x_k < M_k,$ we consider the vector $\delta(\{x_1, \dots, x_k\}) = (\delta_1, \dots, \delta_{k-1}).$ 
Note that $0 \le \delta_i < M_{k-1}$ for all $i \in [k-1]$. For $k = 3,$ the $4$-coloring is given as:
\[
    \phi^{(3)}_m(\{x_1, x_2, x_3\}) = 
    \begin{cases}
        C_1, &\text{if } \delta_1 < \delta_2 \text{ and } \phi_m^{(2)}(\{\delta_1, \delta_2\}) \text{ is red;}\\
        C_2, &\text{if } \delta_1 < \delta_2 \text{ and } \phi_m^{(2)}(\{\delta_1, \delta_2\}) \text{ is blue;}\\
        C_3, &\text{if } \delta_1 \ge \delta_2 \text{ and } \phi_m^{(2)}(\{\delta_1, \delta_2\}) \text{ is red;}\\
        C_4, &\text{if } \delta_1 \ge \delta_2 \text{ and } \phi_m^{(2)}(\{\delta_1, \delta_2\}) \text{ is blue;}\\
        C_1 &\text{if } x=y=z.
    \end{cases}
\]

Let us remark that $\delta_1 = \delta_2$ only when $x = y = z,$ in which case $\delta_1 = \delta_2 = 0$, thus the above definition covers all cases.

If $x_1 < x_k,$ we denote by $\argmax_{i \in [k-1]} \delta_i$ the unique index $j \in [k-1]$ such that $\delta_j = \max_{i \in [k-1]} \delta_i$, where the uniqueness follows from Properties~\ref{prop:not-equal}~and~\ref{prop:maximum}. In the following, a monotone sequence is a non-decreasing or a non-increasing sequence. For $k \ge 4,$ the coloring is given as:

\[ 
    \phi^{(k)}_m(\{x_1, \dots, x_k\}) = 
\begin{cases}
    \phi^{(k-1)}_m(\{\delta_1, \dots, \delta_{k-1}\}), &\text{if } \delta \text{ is a monotone sequence;}\\
    C_1, &\text{if } \delta \text{ is not monotone and } \argmax_{i \in [k-1]} \delta_i \in \{1, k-1\}; \\
    C_2, &\text{if } \argmax_{i \in [k-1]} \delta_i \not\in \{1, k-1\};\\
\end{cases}
\]


Given a positive integer $b$, a $k$-graph $H$ and a mapping $h \colon H \rightarrow [0, \dots, M_k(m)),$ we say that $h$ is an embedding of $H$ into $\phi^{(k)}_m[b]$ if for all $y \in \{0, \dots, M_k(m)-1\},$ we have $|h^{-1}(y)| \le b$. We say that $h$ is an almost monochromatic embedding of $H$ into $\phi^{(k)}_m[b]$ if there is a color $C_q$ such that for all edges $e = \{v_1, \dots, v_k\} \in E(H)$, either the vertices $h(v_1), \dots, h(v_k)$ are not all distinct or $\phi_m^{(k)}(\{h(v_1), \dots, h(v_k)\}) = C_q.$ We say that $h$ is a monochromatic embedding if additionally $\phi^{(k)}(\{h(v_1), \dots, h(v_k)\}) = C_q$ also when $h(v_1), \dots, h(v_k)$ are not all distinct. 

Observe that if there is no monochromatic embedding of $H$ into $\phi^{(k)}_m[b]$, then $R(H; 4) \ge M_k \cdot b$. Indeed, we can define a coloring $\Psi$ of the complete $k$-graph with vertex set $\{0, \dots, M_k-1\} \times [b],$ where a $k$-set $\{(x_1, y_1), \dots, (x_k, y_k)\}$ is colored by $\phi^{(k)}_m(\{x_1, x_2, \dots, x_k\}).$ Then, a monochromatic embedding of $H$ into $\phi^{(k)}_m[b]$ precisely corresponds to a monochromatic copy of $H$ in $\Psi$.

\subsection{Proof outline}\label{subsec:outline}
Our proof borrows some ideas from the work of Graham, R\"{o}dl and Ruci\'{n}ski~\cite{graham2000graphs} who proved the lower bound for graphs. Additionally, as a building block, we use a construction by Brada\v{c}, Fox and Sudakov~\cite{bradavc2023ramsey} of $k$-graphs with $m$ edges whose $4$-color Ramsey number is $\tw_k(\Theta(\sqrt{m}))$.

The $k$-graph $H$ in Theorem~\ref{thm:main} is the edge union of two hypergraphs. The first one, which we shall call $H_R$, where $R$ stands for ``random'', is obtained from the binomial random hypergraph with average degree $O(\Delta)$ by removing high degree vertices. The second one, which we denote by $H_E$, where $E$ stands for ``expander'', is obtained by taking $O(n / \Delta)$ random copies of a $k$-graph $H_0$ on $O(\Delta)$ vertices with maximum degree $O(\Delta)$. The construction of our gadget $H_0$ is a small modification of the construction from~\cite{bradavc2023ramsey} and is obtained by considering a $d$-regular $2$-graph $F$ with expansion properties (here $d$ depends only on $k$), and putting a $k$-edge for any $k-1$ vertices spanning a tree in $F$ and any distinct $k$-th vertex. 

Informally speaking, $H_0$ has two crucial properties. First, any pair of vertices is contained in a $k$-edge. Second, given any set $W\subset V(H_0)$ of size $\varepsilon |V(H_0)|$ (where $\varepsilon>0$ is not too small with respect to $d$), the auxiliary $(k-1)$-graph $H_0'=H_0'(W)$ where $e$ is an edge if $e\cup \{w\}\in E(H_0)$ for some $w\in W$ ``behaves like'' the $(k-1)$-uniform gadget hypergraph we define from $F$. More specifically, since $F$ is an expander, there will be $U\subset V(F)$ of size $(1-\varepsilon)|V(F)|$ so that each $u\in U$ has a neighbor in $W$, whence $H_0'$ contains a $(k-1)$-uniform edge for every $k-2$ vertices in $U$ spanning a tree in $F$, along with any distinct $(k-1)$-th vertex. Consequently, any two vertices in $U$ belong to a common edge of $H_0'$, and we can further iterate this (defining an analogous auxiliary $(k-2)$-graph $H_0''$, it should again behave like our gadget in uniformity $k-2$). The fact that the hypergraph `behaves similarly' at lower uniformities is what allows us to do stepping-up techniques. 

To prove the theorem, we start by assuming there is a monochromatic embedding of $H = H_R \cup H_E$ into $\phi_m^{(k)}[b_k]$, where $m = \Theta_k(\Delta), b_k = \Theta_k(n / \Delta)$ and we aim to reach a contradiction. Using the properties of the stepping up coloring, we would like to show there is a monochromatic embedding of $H'$ into $\phi_m^{(k-1)}[b_{k-1}]$, where $b_{k-1} = O(b_k)$ and $H'$ is a $(k-1)$-uniform hypergraph which has $\Omega(n)$ vertices and still behaves roughly like the original construction $H$. Applying this argument $k-2$ times, we would reach a graph $G$ which should be present monochromatically in the coloring $\phi^{(2)}_m[b_2],$ where $b_2 = \Theta_k(n/\Delta)$. However, the pseudorandom properties of $H_R\subset H$ would imply that $G \in \calG_m$ (see Definition \ref{G_m} below). This would lead to a contradiction as we will show, using Lemma \ref{lem:graham-rodl-rucinski}, that there is no monochromatic copy of such a graph in $\phi^{(2)}_m[b_2],$ where $b_2 = \Omega_k(b_k)$.

Unfortunately, we cannot show there is a monochromatic copy of the desired $(k-1)$-graph $H'$ as above. Rather, we find a hypergraph $H'$ which also consists of a random part $H_R'$ and an expander part $H_E'$ such that there is an embedding $h'$ of $H'$ into $\phi_m^{(k-1)}[b_{k-1}]$, where $h'$ is a monochromatic embedding of $H_E'$ but only an almost monochromatic embedding of $H_R'$. This turns out to be sufficient for our inductive argument. In fact, the expander part $H_E$ allows us to reduce the uniformity by one while keeping the structure as above, while the random part is only used in the last step, where we need to claim that the final graph is not ``almost monochromatically'' present in $\phi^{(2)}_m[b_2].$

\subsection{Technical lemmas}

We shall use the following lemma of Graham, R\"{o}dl and Ruci\'{n}ski.

\begin{lemma}[\cite{graham2000graphs}] \label{lem:graham-rodl-rucinski}
    Let $m \ge 10^9$ and set $s = 2^{10^{-8} m}.$ There is a coloring of $K_s$, represented by $E_R \cup E_B = E(K_s)$, such that for all functions $w \colon [s] \rightarrow [0,1]$ with $\sum_{i=1}^s w(i) = x \ge m$ and any $c \in \{R, B\},$ we have
    \[ W = \sum_{ij \in E_c} w(i) w(j) < 0.51 \binom{x}{2}. \]
\end{lemma}

\begin{defn}
\label{G_m}
    Let $m$ be a positive integer and let $s = 2^{10^{-8}m}$. Let $\calG_m$ be the set of all graphs $G$ satisfying the following. For every partition $V(G) = V_1 \cup \dots \cup V_s$ with $|V_i| \le |V(G)|/m, \forall i \in [s]$, it holds that
    \[ \sum_{i<j \colon e_G(V_i, V_j) > 0} |V_i||V_j| > 0.55 \binom{|V(G)|}{2}. \]
\end{defn}

The reasoning behind the definition of $\calG_m$ is the following lemma.
\begin{lemma} \label{lem:no-graph-almost-mono}
    Let $m$ be a positive integer and let $G \in \calG_m$. Denoting $b = |V(G)| / m,$ there is no almost monochromatic embedding of $G$ into $\phi^{(2)}_m[b]$.
\end{lemma}
\begin{proof}
    Indeed, suppose, for the sake of contradiction, that there is an almost monochromatic copy of $H$ in $\phi = \phi^{(2)}_m[b]$ with color $q \in \{\mathrm{red}, \mathrm{blue}\}.$ In other words, with $M_2 = M_2(m) = 2^{10^{-8}m},$ there is a mapping $h \colon V(G) \rightarrow [0, M_2)$ and a color $q \in \{\mathrm{red}, \mathrm{blue}\}$ such that for any $i \in [0, M_2),$ $|h^{-1}(i)| \le b$ and for any $uv \in E(G),$ either $h(u) = h(v)$ or $\phi^{(2)}_m(h(u), h(v)) = q$.

    For $i \in [0, M_2),$ let $V_i = h^{-1}(i)$ and let $w_i = |V_i| / b.$ Note that $w_i \in [0, 1]$ for all $i \in [0, \dots, M_2)$ and that
    \[ \sum_{i=0}^{M_2-1} w_i = n / b = m. \]
    Thus, by the properties of $\phi$ inherited from Lemma~\ref{lem:graham-rodl-rucinski}, we have that 
    \[ \sum_{0 \le i < j < M_2, \phi(ij) = q} w_i w_j < 0.51 \binom{m}{2}. \]
    Since $h$ is an almost monochromatic embedding in color $q,$ this implies that
    \[ \sum_{0 \le i < j < M_2 \colon e_G(V_i, V_j) > 0} |V_i||V_j| < 0.51 \binom{m}{2} \cdot b^2 \le 0.51 \binom{n}{2}. \]

    On the other hand, since $G \in \calG_m$, we have 
    \[ \sum_{0 \le i < j < M_2, \colon e_G(V_i, V_j) > 0} |V_i||V_j| > 0.55 \binom{n}{2}, \]
    a contradiction.
\end{proof}

\begin{lemma} \label{lem:random-graph-is-good-whp}
    Let $n, m, d$ be positive integers satisfying $m \ge 10^{9}, \, d \ge m, \, n \ge 2^m$. Then, for the random graph $G \sim G(n, d/n)$, we have
    \[ \Pr[G \not\in \calG_m] \le e^{-dn/10}. \]
\end{lemma}
\begin{proof}
    Let $s = 2^{10^{-8}m}$. Fix a partition $V(G) = V_1 \cup \dots \cup V_s$ with $|V_i| \le n/m$ for all $i \in [s].$ Note that the number of pairs of $\binom{[n]}{2}$ in the same set among $V_1, \dots, V_s$ is at most $n^2 / m$, hence if $\sum_{i<j \colon e_G(V_i, V_j) > 0} |V_i||V_j| \le 0.55 \binom{n}{2}$, then $\sum_{i<j \colon e_G(V_i, V_j) = 0} |V_i||V_j| \ge \binom{n}{2} - 0.55 \binom{n}{2} - n^2/m \ge 0.3 \binom{n}{2}$, where we used $m \ge 10^{9}$.

    Taking a union bound over all partitions $V_1 \cup \dots \cup V_s$ and all choices of pairs $(i, j)$ such that $e_G(V_i, V_j) = 0$, we obtain

    \begin{align*}
        \Pr[G \not\in \calG_m] \le s^n 2^{\binom{s}{2}} (1-d/n)^{0.3 \binom{n}{2}} \le s^n 2^{s^2} e^{-0.12 dn} < e^{-dn / 10},
    \end{align*}
    where we used that $d \ge m$, $n \ge 2^m$ and $s=2^{10^{-8}m}$. 
\end{proof}
Next, we introduce some notation for creating lower uniformity hypergraphs. This will be used to analyze our stepping-up process involving our random hypergraph $H_R$.
\begin{defn}
    Let $H$ be a $k$-uniform hypergraph and let $U, W_1, \dots, W_r$ be pairwise disjoint subsets of $V(H)$, where $1 \le r \le k-2$. We define a $(k-r)$-uniform hypergraph $H(U; W_1, \dots, W_r)$ on the vertex set $U$ where a subset $S$ of size $k-r$ of $U$ forms an edge if and only if there are $(w_1, \dots, w_r) \in W_1 \times \dots \times W_r$ such that $S \cup \{w_1, \dots, w_r\} \in E(H)$.
\end{defn}

The following definition states the relevant pseudorandom property of the random part of the construction.

\begin{defn}
    For positive integers $k,m \ge 2$ and real $\alpha \in (0, 1]$, we say that a $k$-uniform hypergraph $H$ is \emph{$(\alpha, m)$-good} if for any pairwise disjoint sets $U, W_1, \dots, W_{k-2} \subseteq V(H)$ with $|U|, |W_1|, \dots, |W_{k-2}| \ge \alpha |V(H)|,$ the graph $H(U; W_1, \dots, W_{k-2})$ is in $\calG_m$.
\end{defn}

Note that if $H$ is an $(\alpha, m)$-good $k$-graph and $U, W$ are disjoint sets vertices of size at least $\alpha |V(H)|,$ then $H(U; W)$ is an $(\alpha', m)$-good $(k-1)$-graph, where $\alpha' = \alpha |V(H)| / |U|$. Similarly, if $H$ is $(\alpha,m)$-good, then the induced subgraph $H[U]$ is $(\alpha',m)$-good with $\alpha' = \alpha\frac{|V(H)|}{|U|}$. 

The following lemma proves, for any $n \ge O(m),$ the existence of an $(\alpha, m)$-good $n$-vertex $k$-graph with maximum degree $O_k(m)$.

\begin{lemma} \label{lem:exists-good-k-graph}
    Let $k \ge 3$ be a given integer and $\alpha \in (0,1].$ Setting $C = \left(\frac{16}{\alpha}\right)^{5k},$ for any $m \ge 10^{9}$, $n \ge C m$ and $\alpha n/2 \geq 2^m$, there exists an $n$-vertex $(\alpha, m)$-good $k$-graph with maximum degree at most $C m.$
\end{lemma}
\begin{proof}
    Set $N = 2n, p = Cm / (4N^{k-1})$, and let $H \sim \calH^{(k)}(N, p)$ be the $N$-vertex binomial random $k$-uniform hypergraph with edge probability $p$. Let $\beta = \alpha/2$ and let us show that with probability at least $1/2$, for any pairwise disjoint sets $U, W_1, \dots, W_{k-2} \subseteq V(H)$ with $|U|, |W_1|, \dots, |W_{k-2}| \ge \alpha |V(H)|$, the graph $H(U; W_1, \dots, W_{k-2})$ is in $\calG_m$. 

    Indeed, consider fixed disjoint sets $U, W_1, \dots, W_{k-2} \subseteq V(H)$ with $|U|, |W_1|, \dots, |W_{k-2}| \ge \beta N$ and let $H' = H(U; W_1, \dots, W_{k-2})$. Denote $T = \prod_{i=1}^{k-2} |W_i|$ and note that $|T|p \le N^{k-2} p \le 1/2.$ For fixed vertices $u, v \in U,$ we have
    \begin{align*}
        \Pr[uv \in E(H')] = \Pr[\Bin(T, p) \ge 1] \ge |T| p - |T|^2 p^2 \ge |T|p / 2 \ge (\beta N)^{k-2} p/2 \ge \frac{8^k m}{\alpha^2 N/2}.
    \end{align*}
    Note that the events $\{ uv \in E(H') \}, \{u,v\} \in \binom{U}{2}$ are mutually independent. Hence, $H'$ is distributed as $G(|U|, p')$ for some $p' \ge \frac{8^k m}{\alpha^2 N/2} \ge \frac{8^k m}{\alpha |U|}$. By Lemma~\ref{lem:random-graph-is-good-whp} it follows that for fixed $U, W_1, \dots, W_{k-2},$ we have 
    \[ \Pr[H' \not\in \calG_m] \le e^{-p'|U|^2/10}\le e^{-8^km|U|/(10\alpha)}\le e^{-8^k m N / 20}. \]
    Using that $m \ge 10^{9}$ and taking a union bound over at most $k^N$ choices for $U, W_1, \dots, W_{k-2},$ we obtain that $H$ is $(\beta, m)$-good with probability at least $3/4$.

    Note that the expected number of edges in $H$ is $\binom{N}{k} p \le \frac{N^k p}{2k} = \frac{CmN}{8k},$ so with probability at least $1/2,$ $e(H) \le \frac{Cm N}{4k}$. Putting it all together, with positive probability $H$ is $(\beta, m)$-good and has at most $\frac{Cm N}{4k}$ edges. Let $H_2$ be the induced subgraph of $H$ on $N/2 = n$ vertices obtained by removing the $N/2$ vertices of largest degree. Recalling that $\beta = \alpha/2,$ since $H$ is $(\beta,m)$-good, it follows that $H_2$ satisfies is $(\alpha, m)$-good. Finally, observe that $\Delta(H_2) \cdot N/2 \le k \cdot e(H),$ which implies $\Delta(H_2) \le \frac{Cm}{2}.$ Putting it all together, with positive probability $H_2$ is the desired hypergraph.
\end{proof}

We proceed to the ``expander'' part of our construction. We shall use the following famous theorem of Friedman~\cite{friedman}, though a much weaker result about the edge distribution of sparse random graphs would work for our purposes. In what follows, for a $d$-regular graph $G$, we denote by $\lambda(G)$ the second largest, in absolute value, eigenvalue of its adjacency matrix.

\begin{theorem}[\cite{friedman}] \label{thm:friedman}
    Fix a real $\eta > 0$ and an even positive integer $d$. Then, as $n$ tends to infinity, for a random $d$-regular $n$-vertex graph $G$, with probability $1 - o(1)$, it holds that $\lambda(G) \le 2 \sqrt{d-1} + \eta$.
\end{theorem}

\begin{lemma} \label{lem:edge-distribution}
    For any integer $k \ge 1$ and any $\varepsilon > 0,$ there are constants $d = d(k, \varepsilon)$ and $M_0 = M_0(k, \varepsilon)$ such that for all $M \ge M_0,$ there is a graph $F$ on $M$ vertices satisfying the following.
    
    \begin{enumerate}[label=F\arabic*)]
        \item \label{prop-large-neighborhood} For any set $U \subseteq V(F)$ such that $|U| \ge \varepsilon M$, the number of vertices with fewer than $k$ neighbors in $U$ is at most $\varepsilon M$.
        \item \label{prop-max-degree} The maximum degree of $F$ is at most $d$.        
    \end{enumerate}
\end{lemma}
\begin{proof}
    Let $d$ be an even integer chosen large enough compared to $k, \varepsilon$. Applying Theorem~\ref{thm:friedman} with $\eta = 1,$ for large enough $M,$ there is an $M$-vertex $d$-regular graph $F$ with $\lambda \coloneqq \lambda(F) \le 2\sqrt{d-1} + 1$. It remains to show that $F$ satisfies \ref{prop-large-neighborhood}. Indeed, let $U$ be an arbitrary set of $\varepsilon M$ vertices of $F$ and let $W$ be the set of vertices with fewer than $k$ neighbors in $U$. By the expander mixing lemma (see for example~\cite[Corollary~9.25]{alon-spencer}), we have
    \[ e_G(U, W) \ge d\frac{|U||W|}{M} - \lambda \sqrt{|U| |W|}. \]
    On the other hand, by definition of $W,$ we have $e_G(U, W) \le k |W|.$ Combining, we have, 
    \[ \lambda \sqrt{|U||W|} \ge (d\varepsilon - k) |W|.\]
    Using $|U| =\varepsilon M,$ $\lambda \le 2\sqrt{d-1} + 1$ and $d$ is large enough, it follows that $|W| \le \varepsilon M,$ as needed.
\end{proof}

As discussed in Subsection~\ref{subsec:outline}, we will essentially place many random copies of a small $k$-uniform hypergraph. The following lemma provides a template for how to place these copies. We shall do this randomly with small alterations.

\begin{lemma} \label{lem:template-hypergraph}
    For any real $\varepsilon > 0,$ there is a constant $C = C( \varepsilon)$ such that for any integers $n, s$ satisfying $n \ge C s^2$ and $s \ge C$, there is a nonempty $s$-uniform $n$-vertex hypergraph $\calT$ satisfying the following:
    \begin{enumerate}[label=P\arabic*)]
        \item $\Delta(\calT) \le C$.
        \item $|e\cap e'|<\varepsilon s$ for all distinct $e,e'\in E(\calT)$.
        \item \label{prop:most-blocks-uncorrelated} For any set $A \subseteq V(\calT)$ with $|A| \ge \varepsilon n,$ there are at most $\varepsilon e(\calT)$ hyperedges $e \in E(\calT)$ such that $\big||e \cap A| - \frac{|A|}{n} \cdot s\big| > \varepsilon s$.
    \end{enumerate}
\end{lemma}
\begin{proof}
    Without loss of generality, assume $\varepsilon < 1/10.$ Let $C = C(\varepsilon)$ be a large constant to be chosen implicitly later and assume that $s \ge C$. Let $K = \frac{\sqrt{C} n}{s}$ and let $X_1, \dots X_K$ be independent multisets each formed by taking a set of $s_0 \coloneqq (1 + \varepsilon / 4) s$ uniformly random elements from $[n]$ chosen with replacement. 

    Let $W \subseteq [n]$ be the set of indices $j$ that lie in more than $C$ sets $X_i, i \in [K]$. For $i \in [K],$ we say that $X_i$ is good if all its elements are distinct, $|X_i\cap X_{i'}|<\varepsilon s$ for all $i'\neq i$, and $|X_i \cap W| \le s_0 - s = \varepsilon s / 4$. Otherwise, $X_i$ is bad. If $X_i$ is good, let $Y_i$ be an arbitrary subset of $X_i \setminus W$ of size $s$ and let $\calT$ be the $s$-uniform hypergraph with vertex set $[n]$ and edge set $\{ Y_i \, \vert \, i \text{ is good}, i \in[K]\}.$ Note that the maximum degree of $\calT$ is at most $C$. We aim to prove that with positive probability $\calT$ is nonempty and satisfies \ref{prop:most-blocks-uncorrelated}.

    Observe that    
    \begin{align*}
        \Pr[|X_i \cap W| > s_0 -  s] &\le \binom{s_0}{s_0 - s} \Pr[\Bin(sK, (s_0-s) / n) \ge C(s_0-s)] \le 2^{s_0} \binom{sK}{C\varepsilon s/4} \left(\frac{\varepsilon s}{4n}\right)^{C \varepsilon s/4}\\
        &\le 2^{s_0} \left( \frac{4eK}{C \varepsilon}\right)^{C \varepsilon s/4} \left(\frac{\varepsilon s}{4n}\right)^{C \varepsilon s/4} = 2^{s_0} \left( \frac{e}{\sqrt{C}}\right)^{C \varepsilon s / 4} < 1 / 24.
    \end{align*}

    The probability that the elements of $X_i$ are not all distinct is at most $\binom{s_0}{2} / n < 1/ 24$. Lastly, the probability $|X_i\cap X_{i'}|\ge \varepsilon s$ for some $i'\neq i$ is at most $K2^{s_0}(s_0/n)^{\varepsilon s}<1/24$. Thus, $X_i$ is good with probability at least $1 - 1 / 8$.

    By Markov's inequality, with probability at least $3/4,$ there are at most $K / 2$ bad sets $X_i$. Observing that $Y_i\neq Y_{i'}$ for distinct good indices $i,i'\in [K]$, we see that $\Pr[e(\calT)\ge K/2] = \Pr[\#(i\in [K]:X_i\text{ is good})\ge K/2]\ge 3/4$.

    Note that if $e(\calT)\ge K/2$, then to have property \ref{prop:most-blocks-uncorrelated}, it is enough that for all $A \subseteq [n]$ of size $|A| \ge \varepsilon n$,
    \begin{equation} \label{eq:lemma-enough}
        \text{ there are at most } \varepsilon K / 4 \text { indices } i \in [K] \text{ such that } \big||X_i \cap A| - \frac{|A|}{n} \cdot s_0 \big| \ge \varepsilon s/4. 
    \end{equation}
    
    Fix $A \subseteq [n]$ with $|A| \ge \varepsilon n$. For fixed $i \in [K],$ we have
    \begin{align*} 
        \Pr\left[\big| |X_i \cap A| - \frac{|A|}{n} \cdot s_0 \big| \ge \varepsilon s / 4 \right] &= \Pr\left[ \Big|\Bin\left(s_0, \frac{|A|}{n}  \right) - \frac{|A|}{n} \cdot s_0 \Big| \ge \varepsilon s / 4\right]\\
        &\le 2 \exp\left( -\frac{ (\varepsilon s/4)^2}{3 \frac{|A|}{n} \cdot s_0}\right) < \exp\left(-\varepsilon^2s / 50 \right),
    \end{align*}

    where we used that $\frac{|A|}{n}\cdot s_0 > \eps s/4$ so that we can apply the Chernoff bound: $\Pr[|\Bin(x,p)- xp|>t]\le 2 \exp(-t^2/ (3 xp))$ which holds for $t\le (3/2) xp$ (e.g. Corollary~2.3~in~\cite{JLRbook}). Thus, the probability that \eqref{eq:lemma-enough} does not hold for $A$ is at most
    \[ \binom{K}{\varepsilon K / 4} \exp\left( -\varepsilon^2 s / 50 \cdot \varepsilon K/4 \right) < 2^K \exp(-\varepsilon^3 sK/200) = 2^K \exp(-\varepsilon^3 \sqrt{C} n / 200) < 4^{-n}. \]

    Taking a union bound, we have that \eqref{eq:lemma-enough} holds for all large sets with probability at least $3/4$.
    
    Thus, with probability at least $1/2$ we have $e(\calT)\ge K/2$ and \eqref{eq:lemma-enough} holds for all relevant sets $A$. As discussed, these conditions imply that $\calT$ satisfies the desired properties, finishing the proof.
\end{proof}

\subsection{The construction}
The main building block of our construction is given in the following definition. It is a variant of the aforementioned construction by Brada\v{c}, Fox and Sudakov~\cite{bradavc2023ramsey}.

\begin{defn}
    For a graph $F$ and $k \ge 2,$ we define a $k$-uniform hypergraph $H = H^{(k)}(F)$ with vertex $V(H) = V(F)$ and 
    \[ E(H) = \{ \{x_1, \dots, x_k\} \, \vert \, F[\{x_1, \dots, x_{k-1}\}] \text{ is connected and } x_1, \dots, x_k \text{ are distinct} \}. \]
\end{defn}
In other words, for any tree on $k-1$ vertices in $F$ and any other vertex, we put a $k$-edge into $H$.

Next we extend this definition by overlaying many of these hypergraphs into a single hypergraph.

\begin{defn}
    Given an integer $n$ and a collection of graphs $\calF$, where $V(F) \subseteq [n], \forall F \in \calF$, we define a hypergraph $H = H^{(k)}(\calF)$ as the $k$-graph with vertex set $[n]$ and edge set $E(H) = \bigcup_{F \in \calF} E(H^{(k)}(F)).$
\end{defn}

For $k \ge 2,$ let $\varepsilon = 10^{-6k}.$ Given $m \ge 10^9$ and $n \ge 2^{m+1} / \varepsilon,$ let $H_R$ be an $(\varepsilon, m)$-good $k$-graph on the vertex set $[n]$ with $\Delta(H_R) \le Cm,$ where $C = \left(16/\varepsilon\right)^{5k},$ whose existence is given by Lemma~\ref{lem:exists-good-k-graph}. Furthermore, let $F$ be a graph on $s = 10^{20k} m$ vertices satisfying \ref{prop-large-neighborhood} with parameter $\varepsilon$ given by Lemma~\ref{lem:edge-distribution}, and note that it has maximum degree bounded by a function of $k$. Let $\calT$ be an $s$-uniform hypergraph on the vertex set $[n]$ given by Lemma~\ref{lem:template-hypergraph} with parameter $\varepsilon$. Let the edges of $\calT$ be $B_1, B_2, \dots, B_{e(\calT)}$. For each $i \in [e(\calT)],$ let $F_i$ be an isomorphic copy of $F$ with vertex set $V(F_i) = B_i$ and let $\calF = \{ F_i \vert \, i \in [e(\calT)]\}$ denote the family of all of these graphs. Let $H_E = H^{(k)}(\calF)$ and finally let $H$ be the $k$-graph with vertex set $[n]$ and $E(H) = E(H_R) \cup E(H_E)$. Note that $\Delta(H) \le C_k m,$ where $C_k$ depends only on $k$.

In the next subsection, we will show that when $H_R,H_E$ exist, that $H$ has four color Ramsey number at least $M_k(m) \cdot 10^{-10k}\cdot (n/m)$. This will quickly imply Theorem~\ref{thm:main}, as we explain below.

\subsection{Proof of Theorem~\ref{thm:main}}\label{subsec:final}
Clearly we may assume that $\Delta$ is sufficiently large in terms of $k$, recalling the statement of Theorem~\ref{thm:main} is trivial unless $\Delta\ge 1/c_k$. Now, let $C_k$ be a sufficiently large constant. Furthermore, we may assume that $n \ge C_k \cdot \Delta,$ for some large constant $C_k$ as otherwise we may apply the argument with $\Delta' = \Delta / C_k$ at the cost of decreasing the constant $c_k$ in the result (recall that we initially assume $n\ge \Delta$). If $n \le 2^{C_k \Delta},$ then by a result of Brada\v{c}, Fox and Sudakov~\cite{bradavc2023ramsey}, there is a $k$-graph with at most $C_k \Delta$ vertices, maximum degree at most $\Delta,$ and $4$-color Ramsey number at least $\tw_k(c_k' \Delta) > \tw_k(c_k''\Delta) \cdot n,$ where $c_k', c_k''$ are positive constants depending only on $k$.
    
So we may assume that $n \ge 2^{C_k \Delta}$ for an arbitrary constant $C_k$ depending only on $k$. Let $H$ be the $k$-graph constructed in the previous subsection with $m = (\varepsilon/16)^{5k}\Delta$ (this will exist, as we now have $n\ge 2^{m+1}/\varepsilon$). We shall prove that there is no monochromatic embedding of $H$ into $\phi_m^{(k)}[b_k],$ where $b_k = (10^5)^{-2k+4} \cdot n/m.$ When $m$ is large with respect to $k$, we have $M_k(m)\ge \tw_k(10^{-10} m)$.

Consequently, this would imply Theorem~\ref{thm:main}, since $M_k(m) \cdot b_k > \tw_k(10^{-10} m) \cdot (10^5)^{-2k+4} n / m > \tw_k(c_k \Delta) \cdot n,$ and $\Delta(H) \le C_k' m,$ where the constants $c_k, C_k' > 0$ depend only on $k$. Suppose, for the sake of contradiction, that there is a monochromatic embedding of $H$ into $\phi_m^{(k)}[b_k]$. Our proof proceeds by induction on the uniformity. The step of the induction is given by the following key claim.

\begin{claim} \label{claim:induction}
    Let $u \in [2,k]$ be an integer and denote $\alpha_u = (10^5)^{-k+u}$. Then, there are disjoint sets $U^u, W_{u+1}, \dots, W_k \subseteq [n]$ and sets $B_1^u, \dots, B^u_{[e(\calT)]}$ such that $B_i^u \subseteq B_i \cap U^u$ for all $i \in [e(\calT)]$ satisfying the following.
    
    \begin{enumerate}[label=\alph*)]
        \item \label{claim:induction-point1} $|U^u| = \alpha_u n$ and $|W_i| = \alpha_{i-1} n$, for all $i \in [u+1, k].$
        \item \label{claim:induction-point2} For every $i \in [e(\calT)],$ either $B^u_i = \emptyset$ or $|B_i \cap U^u| \ge (1-\varepsilon) \alpha_u s$ and $|(B_i \setminus B_i^u) \cap U^u| \le 2(k-u) \varepsilon s$. Furthermore, $|\{i \, \vert \, B^u_i = \emptyset\}| \le 4 (k-u) \varepsilon e(\calT).$
        \item \label{claim:induction-point3} Denote $H_R^u \coloneqq H(U^u; W_{u+1}, \dots, W_k)$, $H_E^u \coloneqq H^{(u)}( \{F_i[B^u_i] \, \vert \, i \in [e(\calT)]\})$, and $b_u \coloneqq (10^5)^{-k -u+4} \cdot n/m$ and $\phi^u = \phi_m^{(u)}[b_u]$. Then, there is a mapping $h^u \colon U^u \rightarrow [0, M_u(m))$ such that $h^u$ is an almost monochromatic embedding of $H_R^u$ into $\phi^u[b_u]$ and $h^u$ is a monochromatic embedding of $H_E^u$ into $\phi^u[b_u]$.
    \end{enumerate}
\end{claim}

Before proving Claim~\ref{claim:induction}, let us show that it leads to a contradiction. Indeed, let $U, W_3, \dots, W_k$ be the sets given by Claim~\ref{claim:induction} with $u = 2$ such that for $H_R^2 = H(U; W_3, \dots, W_k)$, there is an almost monochromatic embedding of $H_R^2$ into $\phi_m^{(2)}[b_2]$. Recalling that $H_R$ is $(10^{-5k}, m)$-good and that $|U|, |W_3|, \dots, |W_k| \ge 10^{-5k} \cdot n$, it follows that $H_R^2 \in \calG_m$. Thus, Lemma~\ref{lem:no-graph-almost-mono} implies that there is no almost monochromatic embedding of $H_R^2$ into $\phi_m^{(2)}[|V(H_R^2)| / m]$, which is a contradiction since $|V(H_R^2)| / m = |U| / m = \alpha_2 n / m = b_2.$

\begin{proof}[Proof~of~Claim~\ref{claim:induction}]
    We prove the claim by reverse induction on $u$. For $u = k$, the claim is given by our original assumption by taking $U^k = [n], B^{k}_i = B_i,$ for all $i \in [e(\calT)]$ and $H^{k}_R = H_R$.

    Now, assume that $3 \le u \le k$ and we are given sets $U^u, W_{u+1}, \dots, W_k$, sets $B_i^u,$ for $i \in [e(\calT)],$ as well as the embedding $h^u$ satisfying the claim for $u$. We wish to prove the claim for $u-1$.

    Without loss of generality, we shall assume that $U^u = [n_u]$, where $n_u = (10^5)^{-k+u} \cdot n$ and that the mapping $h^u$ is given by $h^u(i) = x_i$, such that $0 \le x_1 \le x_2 \le \dots \le x_{n_u} < M,$ where $M = M_u(m)$ was previously defined.

    For two sets of integers $A$ and $A'$ we write $A < A'$ if $a < a'$ for all $a \in A, a' \in A'$.
        
    For $i \in [n_u-1],$ let $\delta_i = \delta(x_i, x_{i+1}).$ We shall repeatedly use the following simple facts. Given an interval $[\ell, r] \subseteq [n_u-1]$ with $r - \ell > b_u,$ let us denote $\delta^* = \max_{i \in [\ell, r]} \delta_i.$ Then, there is a unique index $i^*$ such that $\delta_{i^*} = \delta^*.$ Indeed, since $x_\ell < x_{r+1},$ we have $\max_{i \in [\ell,r]}\delta_i = \delta(x_\ell, x_{r+1}) \ge 1$. If there are two indices $i_1, i_2 \in [\ell, r], i_1 < i_2$ with $\delta_{i_1} = \delta_{i_2} = \delta^*,$ then using Property~\ref{prop:maximum}, we have $\delta^* = \delta(x_{\ell}, x_{i_1+1}) = \delta(x_{i_1+1}, x_{r+1}) = \delta(x_{\ell}, x_{r+1})$, which contradicts Property~\ref{prop:not-equal}. Furthermore, given an interval $J \subseteq [\ell, r]$ with $i^* \not\in J,$ we clearly have $\max_{j \in J} \delta_j < \delta^*$.    
    
    We run the following procedure in steps $j = 1, \dots, T$. At each step, we have two parameters $\ell_j, r_j$ with $1 \le \ell_j < r_j \le n_u$. If $r_j - \ell_j < n_u/2000$, we stop the process and declare our total number of steps $T$ to be $j$. Otherwise, we shall consider the interval $x_{\ell_j}, \dots, x_{r_j}.$ We start with $\ell_1 = 1, r_1 = n_u$.

    At step $j$, let $p_j$ denote the unique index such that $\delta_{p_j} = \max\{ \delta_i \, \vert \, i \in [\ell_j, r_j-1]\}$. Since for $u \geq 3$, $b_u \le 10^{-5} n_u < n_u/2000,$ we have $r_j > \ell_j + b_u,$ so $p_j$ really is unique by the discussion above. If $p_j - \ell_j \ge r_j - p_j$, we call $j$ a \emph{left step} and set $\ell_{j+1} = \ell_j, r_{j+1} = p_j$, and otherwise we call $j$ a \emph{right step} and set $\ell_{j+1} = p_j+1, r_{j+1} = r_j.$ We proceed to step $j+1$. Note that at every step the length of the new interval $[\ell_{j+1}, r_{j+1}]$ is at least half of the length of the previous interval $[\ell_j, r_j]$. In particular $r_T - \ell_T \ge n_u/4000$.

    We will use the following simple fact.
    \begin{claim} ~\label{cl:find-two}
        At least one of the following holds.
        \begin{enumerate}[label=C\arabic*)]
            \item \label{case:right-steps}
            There are indices $t_1, t_2, t_3$ with $1 < t_1 < t_2 < t_3 \le T$ such that $t_1-1$, $t_2-1$ and $t_3-1$ are right steps and $\ell_{t_1} - \ell_1, \ell_{t_2} - \ell_{t_1}, \ell_{t_3} - \ell_{t_2} \ge n_u/1024$.
            \item \label{case:left-steps} There are indices $t_1, t_2, t_3$ with $1 < t_1 < t_2 < t_3 \le T$ such that $t_1-1$, $t_2-1$ and $t_3-1$ are left steps and $r_1 - r_{t_1}, r_{t_1} - r_{t_2}, r_{t_2} - r_{t_3} \ge n_u/1024$.    
        \end{enumerate}        
    \end{claim}
    \begin{proof}
        Let $a_0 = 1, b_0 = n_u$ and for $i \in [5],$ let $t_i'$ be the minimum index $j \in [T]$ such that $\ell_j - a_{i-1} \ge n_u / 1024$ or $b_{i-1} - r_j \ge n_u / 1024,$ where we will show the existence of such an index shortly. Set $a_i = \ell_{t_i'}$ and $b_i = r_{t_i'}$. Observe that if $a_i - a_{i-1} \ge n_u / 1024,$ then $t_i'-1$ is a right step and if $b_i - b_{i-1} \ge n_u / 1024,$ then $t_i'-1$ is a left step. Furthermore, by the minimality of $t_i',$ note that $b_i - a_i \ge (b_{i-1} - a_{i-1} - 2\cdot n_u/1024) / 2.$

        A simple induction shows that for $i \in [5],$ we have $b_i - a_i +1\ge (b_{i-1} - a_{i-1} + 1) / 4.$ Indeed, if $i \in [5],$ and the above holds for all $i' < i,$ then $b_{i-1} - a_{i-1} + 1\ge n_u / 4^{i-1} \ge n_u / 256$. Since $n_u/256 > 2 \cdot n_u / 1000 + n_u / 2000,$ $t_i'$ is well-defined and we have $b_i - a_i + 1 \ge 1/2 + (b_{i-1} - a_{i-1} + 1 - 2\cdot n_u/1024) / 2 \ge (b_{i-1} - a_{i-1} + 1) / 4,$ as needed.

        Finally, by the pigeonhole principle, there are either three indices $i_1, i_2, i_3$ with $1 \le i_1 < i_2 < i_3 \le 5$ such that for all $j \in [3],$ we have $a_{i_j} - a_{i_j-1} \ge n_u/1024$ or for all $j \in [3],$ we have $b_{i_j}  - b_{i_j-1} \ge n_u/1024$. For $j \in [3],$ taking $t_j = t_{i_j}'$, in the former case these indices satisfy~\ref{case:right-steps}, whereas in the latter case, they satisfy~\ref{case:left-steps}, finishing the proof of the claim.
    \end{proof}
    Since the two cases in Claim~\ref{cl:find-two} are symmetric, without loss of generality, we shall assume that there are indices $1 < t_1 < t_2 < t_3 \le T$ such that $\ell_{t_1} - \ell_1, \ell_{t_2} - \ell_{t_1}, \ell_{t_3} - \ell_{t_2} \ge n_u/1024$.

    Let $I_1 = [1, \ell_{t_1}-1], I_2 = [\ell_{t_1}, \ell_{t_2}-1]$ and $I_3 = [\ell_{t_2}, \ell_{t_3}-1].$ Furthermore, let $R \subseteq [t_1, t_2)$ denote the set of right steps between $t_1$ and $t_2$. For each $j \in R,$ let $S_j = [\ell_j, p_j]$ and note that $S_j < S_{j'}$ for $j < j'$ and $\bigcup_{j \in R} S_j = I_2$.

    For a set $A \subseteq [n_u]$ and $i \in [e(\calT)]$, we say the sets $A$ and $B_i$ are \emph{correlated} if $\big| |A \cap B_i| - \frac{|A|}{n} \cdot s| \ge \varepsilon s.$ By~\ref{prop:most-blocks-uncorrelated}, for any set $A$ of size at least $\varepsilon n,$ there are at most $\varepsilon e(\calT)$ indices $i$ such that $A$ and $B_i$ are correlated.
    
    We shall now do an initial round of `pruning' to our collection of blocks $\{B_i^u:i\in [e(\mathcal{T})]\}$. If $B_i$ is correlated with $I_j$ for some $j \in [3],$ let $B_i' = \emptyset.$ Otherwise, recalling that we have placed a copy $F_i$ of $F$ on the vertex set $B_i$, let $B_i'$ be the set of all vertices in $B^u_i \cap I_2$ that in the graph $F_i$ have at least $k$ neighbors in each of the sets $B^u_i \cap I_1$ and $B^u_i \cap I_3$. 

    Assume that $i \in [e(\calT)]$ is such that $B^u_i \neq \emptyset$ and for all $j \in [3],$ the sets $B_i$ and $I_j$ are not correlated. Then, for $j \in [3],$
    \[ |B^u_i \cap I_j| = |B_i \cap I_j| - |(B_i \setminus B^u_i) \cap I_j| \ge \left(\frac{|I_j|}{n} s -\varepsilon s\right) - 2(k-u) \varepsilon s \ge \left(\frac{\alpha_u}{1024} - (2(k-u)+1) \varepsilon\right) s > \varepsilon s, \]
    where we used that $B_i$ is not correlated with $I_j$. By~\ref{prop-large-neighborhood}, there are at most $2 \varepsilon s$ vertices in $B^u_i \cap I_2$ with fewer than $k$ neighbors in $I_1$ or fewer than $k$ neighbors in $I_3$. Hence, 
    \begin{equation} \label{eq:diff-in-I2}
        |(B_i \setminus B_i') \cap I_2| \le |(B_i \setminus B_i^u) \cap U^u| + 2 \varepsilon s \le 2(k-u+1) \varepsilon s.        
    \end{equation}
    In particular,     
    \[ |B_i'| = |B_i \cap I_2| - |(B_i \setminus B_i') \cap I_2| \ge  \frac{|I_2|}{n} s-\varepsilon s - 2(k-u+1) \varepsilon s > 10^{-5}\alpha_u s > \varepsilon s.
    \]
    
    Thus for $i\in [e(\mathcal{T})]$ we have shown that $B_i' = \emptyset$ if and only if either $B_i^u=\emptyset$ (which happens for at most $4(k-u)\varepsilon e(\mathcal{T})$ indices by assumption) or $B_i$ is correlated with $I_j$ for some $j\in [3]$ (which happens for at most $3\varepsilon e(\mathcal{T})$ total indices). Together we get
    \begin{equation}\label{eq:many nonempty}
        |\{i\in [e(\mathcal{T})]: B_i'=\emptyset\}|\le (4(k-u)+3)\varepsilon e(\mathcal{T}).
    \end{equation}
    
    \begin{claim} \label{cl:no-two-in-same-interval}
        For any $i \in [e(\calT)]$ and any $j \in R,$ it holds that $|B'_i \cap S_j| \le 1.$
    \end{claim}
    \begin{proof}
        To simplify notation, in the proof of this claim we write $h = h^u$ and $\phi = \phi^u$.
        
        Assume, for the sake of contradiction, that $|B'_i \cap S_j| \ge 2,$ for some $i \in [e(\calT)]$ and $j \in R.$ Let $v, w$ be two distinct elements in $B'_i \cap S_j$ such that $v < w$. Recall that, by definition of $B'_i,$ each of $v$ and $w$ has at least $k$ neighbors in $I_1$ and in $I_3$ in the graph $F_i$. Using this, we exhibit two $k$-edges of $H_E^u$ which are mapped by $h$ to edges colored differently under $\phi$, contradicting our assumption.

        By Property~\ref{prop:maximum}, the following hold (see Figure~\ref{fig:intervals} for an illustration):
        \begin{enumerate}[label=D\arabic*)]
            \item \label{D1} $\delta(x_v, x_w) < \delta_{p_j} < \delta_{\ell_{t_1}-1}$;
            \item \label{D2} for any $y \in I_1$, $\delta(x_y, x_v) \ge \delta_{\ell_{t_1}-1}$;
            \item \label{D3} for any $z \in I_3,$ $\delta(x_w, x_z) = \delta_{p_j} \in [\delta_{\ell_{t_2-1}}, \delta_{\ell_{t_1-1}})$;
            \item \label{D4} for any $z, z' \in I_3$, $\delta(x_z, x_{z'}) < \delta_{\ell_{t_2-1}}$.
        \end{enumerate}

        \begin{figure}[h]
            \centering
            \includegraphics[width=0.8\textwidth]{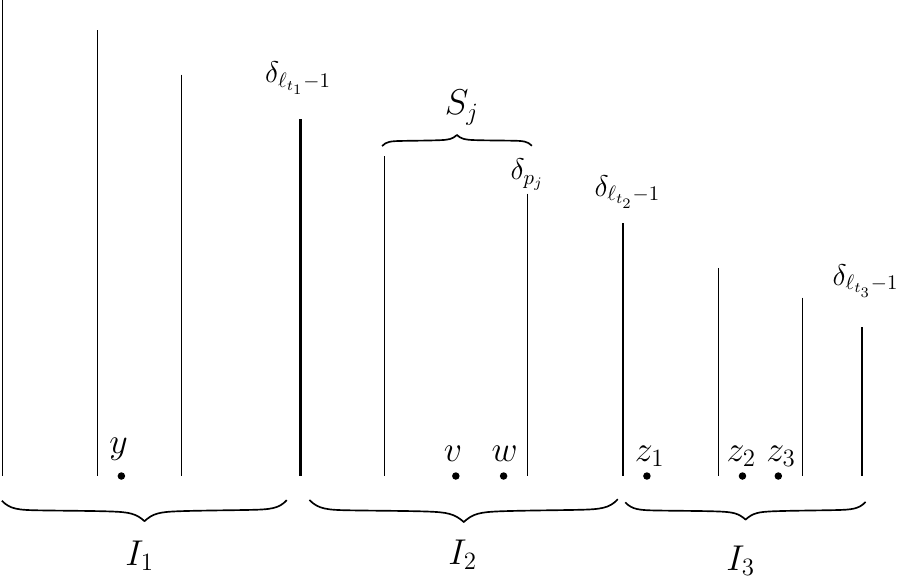}
            \caption{A picture representing the intervals $I_1, I_2, I_3$ and the relevant vertices. Similarly as in Figure~\ref{fig:deltas}, the vertical lines correspond to $\delta_i$ for some $i \in [n_u-1]$.}
            \label{fig:intervals}
        \end{figure}
        
        First, consider the case $u = 3$. Recall that there are $y \in I_1$ and $z \in I_3$ such that $vy, vz \in E(F_i)$ and thus $e_1 \coloneqq \{y,v,w\}$ and $e_2 \coloneqq \{ v, w, z\}$ are both edges of $H_E^u$. By \ref{D1} and \ref{D2}, we have $\delta(x_y, x_v) \ge \delta_{\ell_{t_1}-1} > \delta(x_v,x_w)$, which implies that $\phi(h(e_1)) = \phi(\{h(y), h(v), h(w)\}) = \phi(\{x_y, x_v, x_w\}) \in \{ C_3, C_4\}.$ On the other hand, by \ref{D1} and \ref{D3}, $\delta(x_v, x_w) < \delta_{p_j} = \delta(x_w, x_z)$, implying $\phi(h(e_2)) = \phi(\{h(v), h(w), h(z)\}) = \phi(\{x_v, x_w, x_z\}) \in \{ C_1, C_2 \}.$ This contradicts our assumption that $h$ is a monochromatic embedding of $H_E^u$ into $\phi$.
        
        Now, consider the case $u \ge 4$. Similarly, there are $y \in I_1$ and $z_1, \dots, z_{k-2} \in I_3$, where $z_1 < z_2 < \dots < z_{k-2}$, such that $vy, vz_1, vz_2, \dots, vz_{k-2} \in E(F_i)$. Thus $e_1 \coloneqq \{ y, v, w, z_1, \dots, z_{k-3}\}$ and $e_2 \coloneqq \{ v, w, z_1, \dots, z_{k-2}\}$ are both edges of $H_E^u$. By \ref{D2}, we have $\delta(x_y, x_v) \ge \delta_{\ell_{t_1-1}}.$ By \ref{D1} and \ref{D3}, we have $\delta(x_v, x_w) < \delta_{p_j} = \delta(x_w, x_{z_1}).$ By \ref{D4}, for all $j \in [k-3],$ $\delta(z_j, z_{j+1}) < \delta_{\ell_{t_2-1}} \le \delta_{p_j} < \delta_{\ell_{t_1}-1}.$ All of this implies that $\phi(h(e_1)) = \phi(\{x_y, x_v, x_w, x_{z_1}, \dots, x_{z_{k-3}}\}) = C_1$ (indeed, writing $\delta_1':= \delta(x_y,x_v),\delta_2'= \delta(x_v,x_w),\dots$, this is a non-monotone sequence with $\argmax_{i\in [k-1]}\delta_i' =1$, corresponding to $\delta_1'=\delta(x_y,x_v)$). On the other hand, the same observations imply that $\phi(h(e_2)) = \phi(\{x_v, x_w, x_{z_1}, \dots, x_{z_{k-2}}\}) = C_2$ (again we get a non-monotone sequence $\delta_1',\dots,\delta_{k-1}'$, now with $\argmax_{i\in [k-1]}\delta'_i =2$, corresponding to $\delta_2'=\delta(x_w,x_{z_1})$). Again, we have reached a contradiction.
    \end{proof}    

    Recall that $b_{u-1} = (10^5)^{-k-u+3} \cdot n/m$. Let $R_S = \{ j \in R \, \vert \, |S_j| \le b_{u-1} \},$ be the set of small intervals and $R_B = \{ j \in R \, \vert \, |S_j| > b_{u-1} \}$ be the set of big intervals.

    Suppose that $\sum_{j \in R_B} |S_j| \ge n_u / 10^4.$ Denote $S = \bigcup_{j \in R_B} S_j$. By the properties of $\calT$ given by Lemma~\ref{lem:template-hypergraph}, there are at least $(1 - \varepsilon) e(\calT)$ indices $i$ such that $\big| |B_i \cap S| - |S|/n \cdot s \big| \le \varepsilon s.$ Thus, there is an index $i$ such that $B'_i \neq \emptyset$ and $\big| |B_i \cap S| - |S|/n \cdot s \big| \le \varepsilon s$. It follows that
    \[ |B'_i \cap S| \ge |B_i \cap S| - |(B_i \setminus B_i') \cap U| \ge (|S| / n -\varepsilon) \cdot s - 2(k-u+1)\varepsilon \cdot s > \alpha_u s / (2\cdot 10^4). \]
    
    Observe that $|R_B| \le n_u / b_{u-1} = (10^5)^{2u-3} m <10^{10k}m< |B'_i|,$ since $|B'_i| > \alpha_u s / 20000 > 10^{20k-5k-5}m$. By the pigeonhole principle, for some $j \in R_B,$ it holds that $|B'_i \cap S_j| \ge 2,$ contradicting Claim~\ref{cl:no-two-in-same-interval}.

    Therefore, we have $\sum_{j \in R_B} |S_j| \le n_u / 10^4$ and thus $\sum_{j \in R_S} |S_j| \ge |I_2| - n_u / 10^4\ge n_u / 10^4.$ Let $U^{u-1}$ be an arbitrary subset of $ \bigcup_{j \in R_S} S_j$ of size $n_{u-1}$ and let $W_u$ be an arbitrary subset of $I_3$ of size $n_{u-1}$, which exists since $n_{u-1} = 10^{-5} n_u$.
    
     For $i \in [e(\calT)],$ let
    \[ B^{u-1}_i =
        \begin{cases}
            B'_i \cap U^{u-1} &\text{if } U^{u-1} \text{ and } B_i \text{ are not correlated};\\
            \emptyset &\text{otherwise.}
        \end{cases}
    \]

    Note that if $B_i' \neq \emptyset$ and $U^{u-1}$ and $B_i$ are not correlated, then $|B_i \cap U^{u-1}| \ge \alpha_{u-1} s - \varepsilon s$ and 
    \[
        |(B_i \setminus B_i^{u-1}) \cap U^{u-1}| \le |(B_i \setminus B_i') \cap I_2| \le 2(k-u + 1) \varepsilon s\](recall (\ref{eq:diff-in-I2})).
    Moreover, there are at most $\varepsilon e(\calT)$ indices $i\in [e(\calT)]$ which are correlated with $U^{u-1}$, and at most $(4k+3)\varepsilon e(\calT)$ indices $i\in [e(\calT)]$ where $B_i'=\emptyset$ (recall ($\ref{eq:many nonempty}$)). So the above holds for at least $\big(1 - 4(k-u+1) \varepsilon\big) e(\calT)$ of the indices $i \in [e(\calT)],$ proving~\ref{claim:induction-point2} for $u-1$.


    We have defined the sets $U^{u-1}, W_u, \dots, W_k$ and $(B_i^{u-1})_{i \in [e(\calT)]}$ and shown they satisfy ~\ref{claim:induction-point1}~and~\ref{claim:induction-point2}. It remains to show~\ref{claim:induction-point3}.
    Denote $j^* = \ell_{t_2}$ and $x^* = x_{j^*}$. We define the embedding $h^{u-1} \colon H^{u-1} \rightarrow [0, M_{u-1}(m) - 1]$ as follows. For $i \in U^{u-1},$ let 
    \begin{equation} \label{eq:new-embedding}
       h^{u-1}(i) \coloneqq \delta(x_i, x^*) = \max_{z \in [i, j^*)} \delta_z,
   \end{equation}
   where the equality holds by~\ref{prop:maximum}. Note that indeed, $h^{u-1}$ only takes values in $[0,M_{u-1}(m)-1]$. This is because $x_i,x^*$ both belong to $[0,M_u(m)-1]= [0,2^{M_{u-1}(m)-1}-1]$, combined with the fact that $\delta(a,b)\le \lceil \log_2(\max\{a,b\}+1)\rceil$ for non-negative integers $a,b$.
    
    Observe that
    
    \begin{equation} \label{eq:deltas-from-Sj}
    j \in R_S, i \in S_j \implies \delta(x_i, x^*) = \max_{z \in [i, \ell_{t_2})} \delta_z \in [\delta_{p_j}, \delta_{\ell_j-1}-1].
    \end{equation}
    
    Therefore, $h^{u-1}(i) > h^{u-1}(i')$ for any $i \in S_j, i' \in S_{j'}$, where $ j,j' \in R_S$ and $j<j'.$ This implies that $h^{u-1}$ maps at most $b_{u-1}$ elements to the same value.
    
    It remains to show that $h^{u-1}$ is an almost monochromatic embedding of $H_R^{u-1}$ and a monochromatic embedding of $H_E^{u-1}$ into $\phi^{u-1}_m[b_{u-1}].$

    We will first show the former. Suppose that $h^{u}(H_R^{u})$ is almost monochromatic in color $C_q$. Recall that $H_R^{u-1}$ is the $(u-1)$-uniform hypergraph obtained by taking any $(u-1)$-set in $U^{u-1}$ that can be extended to an edge of $H_R$ by adding one vertex from each of $W_u, \dots, W_k$. Equivalently, a set $e \in \binom{U^{u-1}}{u-1}$ is an edge of $H_R^{u-1}$ if and only if there is $v \in W_u$ such that $e \cup \{v_e\} \in E(H_R^u)$. Recall that $W_u = I_3.$ Consider an edge $e = \{v_1, \dots, v_{u-1}\} \in E(H_R^{u-1}),$ and a corresponding vertex $v_u \in I_3$ such that $e \cup \{v_u\} \in E(H_R^u)$. Without loss of generality, we have $v_1 < v_2 < \dots < v_u$. From~\eqref{eq:new-embedding}, it follows that $h^{u-1}(v_1) \ge h^{u-1}(v_2) \ge \dots \ge h^{u-1}(v_{u-1})$. We may assume that $h^{u-1}(v_1) > h^{u-1}(v_2)> \dots > h^{u-1}(v_{u-1}),$ otherwise there is nothing to prove since we only aim to show that $h^{u-1}$ is an almost monochromatic embedding of $H_R^{u-1}$.

    Consider $i \in [u-2]$. Since 
    \[ \max_{z \in [v_i, j^*)} \delta_z = h^{u-1}(v_i) > h^{u-1}(v_{i+1}) = \max_{z \in [v_{i+1}, j^*)} \delta_z, \]
    it follows that
    \[ \delta(x_{v_i}, x_{v_{i+1}}) = \max_{z \in [v_i, v_{i+1})} \delta_z = \max_{z \in [v_i, j^*)} \delta_z = h^{u-1}(v_i). \]
    Similarly, 
    
    \[ \delta(x_{v_{u-1}}, x_{v_u}) = \max_{z \in [v_{u-1}, v_u)} \delta_z = \max_{z \in [v_{u-1}, j^*)} \delta_z = h^{u-1}(v_{u-1}), \]
    since $\delta_z < \delta_{j^*}$ for all $z \in [\ell_{t_2}, \ell_{t_3})$. 
    %
    Therefore, $\delta(x_{v_1}, x_{v_2}) > \delta(x_{v_2}, x_{v_3}) > \dots > \delta(x_{v_{u-1}}, x_{v_u})$, so
    \[ \phi^{(u)}(\{x_{v_1}, \dots, x_{v_u}\}) = \phi^{(u-1)}( \{\delta(x_{v_1}, x_{v_2}), \delta(x_{v_2}, x_{v_3}),  \dots , \delta(x_{v_{u-1}}, x_{v_u})\}), \]
    if $u \ge 4$ and if $u = 3,$ then
    
    \[ \phi^{(3)}(\{x_{v_1}, x_{v_2}, x_{v_3}\}) =
        \begin{cases}
            C_3, &\text{if } \phi^{(2)}( \{\delta(x_{v_1}, x_{v_2}), \delta(x_{v_2}, x_{v_3})\}) = \mathrm{red};\\
            C_4, &\text{if } \phi^{(2)}( \{\delta(x_{v_1}, x_{v_2}), \delta(x_{v_2}, x_{v_3})\}) = \mathrm{blue}.
        \end{cases}
    \]

    Recalling that $\delta(x_{v_i}, x_{v_{i+1}}) = h^{u-1}(v_i)$ for all $i \in [u-1]$ (assuming that $h^{u-1}(v_1)>\dots >h^{u-1}(v_{u-1})$), we get that in either case, the color $\phi^{(u)}(h^u(e\cup \{v_u\}))$ determines the color $\phi^{(u-1)}(\{h^{u-1}(v_1), \dots, h^{u-1}(v_{u-1})\})=\phi^{(u-1)}(h^{u-1}(e))$. To finish, we shall argue that $\phi^{(u)}(h^u(e\cup \{v_u\}))=C_q$, which will imply $h^{u-1}$ is indeed an almost monochromatic embedding of $H_R^{u-1}$ into $\phi^{(u-1)}_m[b_{u-1}]$.
    
    Since $h^u$ is an almost monochromatic embedding of $H_R^u$ in color $C_q$, it suffices to show that $h^u$ is injective on the set $e\cup \{v_u\}$ (if it was not injective, then $\phi^{(u)}(h^u(e\cup \{v_u\}))$ is still well-defined but might not be $C_q$). If this was not the case, then there would be some $i\in [u-1]$ so that $\delta(x_{v_i},x_{v_{i+1}})=0$. But we have $\delta(x_{v_{u-1}},x_{v_u}) \ge \delta_{\ell_{t_2}-1}>\delta_{\ell_{t_3}-1}\ge 0$ and since we already observed $\delta(x_{v_1},x_{v_2})>\dots>\delta(x_{v_{u-1}},x_{v_u})$, this completes the argument.

    It remains to show that $h^{u-1}$ is a monochromatic embedding of $H_E^{u-1}$ into $\phi^{(u-1)}_m[b_{u-1}]$. Recall that $h^u$ is a monochromatic embedding of $H_E^u$ into $\phi^u[b_u]$ and suppose the color of this embedding is $C_{q'}$. Consider an edge $e = \{v_1, \dots, v_{u-1}\} \in E(H_E^{u-1})$ and without loss of generality, assume that $v_1 \le v_2 \le \dots \le v_{u-1}$. By definition, there is an index $i \in [e(\calT)]$ such that $e \subseteq B^{u-1}_i$ and some $(u-2)$ vertices of $e$ form a connected set in $F_i[e]$. Recalling that every vertex in $B^{u-1}_i$ has a neighbor in $I_3\cap B_i^u$, there is a vertex $v_u \in I_3 \cap B_i^u$ such that some $u-1$ vertices of $e \cup \{v_u\}$ form a connected set in $F_i[e \cup \{v_u\}]$, implying that $e \cup \{v_u\} \in E(H_E^u)$. By Claim~\ref{cl:no-two-in-same-interval}, no two vertices of $e$ lie in the same set $S_j$. From \eqref{eq:deltas-from-Sj}, it follows that $h^{u-1}(v_1) > h^{u-1}(v_2) > \dots > h^{u-1}(v_{u-1})$. By the same argument as above, it follows that the color $\phi^{(u-1)}(h^{u-1}(e))$ is determined by $\phi^{(u)}(h^u(e\cup \{v_u\})) =C_{q'}$, implying that $h^{u-1}$ is a monochromatic embedding of $H_E^{u-1}$ into $\phi^{(u-1)}_m[b_{u-1}]$, as needed.
\end{proof}

\vspace{0.2cm}
{\bf Acknowledgment:}  We would like to thank David Conlon for stimulating discussions and valuable comments on an earlier draft of this paper.

\bibliographystyle{plain}
\bibliography{references}

\end{document}